\documentclass[12pt]{amsart}
\usepackage{epsfig}
\usepackage{mathtools}
\usepackage[subnum]{cases}
\usepackage{enumerate}
\usepackage{bbm,bm}
\usepackage{amsmath,amssymb,mathrsfs,amsthm}
\usepackage{color}
\usepackage{verbatim}
\usepackage{tikz}
\usepackage{multicol}
\usepackage{multirow}
\usepackage{wrapfig}
\usepackage[hang]{footmisc} 

\usepackage{pgf,tikz,pgfplots,tikz-3dplot-circleofsphere,tkz-euclide}
\pgfplotsset{compat=1.14}
\usetikzlibrary{arrows}

\newtheorem{theorem}{Theorem}[section]
\newtheorem{definition}[theorem]{Definition}
\newtheorem{lemma}[theorem]{Lemma}
\newtheorem{corollary}[theorem]{Corollary}

\newtheorem{remark}[theorem]{Remark}

\newtheorem{problem}[theorem]{Problem}

\newtheorem*{property*}{Property}

\newcommand{\R}{\mathbb{R}} 
 
\newcommand{\Sp}{\mathbb{S}^{n-1}}

\DeclareMathOperator{\dist}{dist}
\DeclareMathOperator{\spann}{span}
\DeclareMathOperator{\vol}{vol}

\begin{document}
\setcounter{footnote}{0}

 \title[Extremal arrangements of points in the sphere]{Extremal arrangements of points in the sphere for weighted cone-volume functionals}

    \author{Steven Hoehner and Jeff Ledford}
    \date{\today}

	\subjclass[2020]{Primary: 52A40; Secondary 52A38; 52B11; 52C99}
	\keywords{Jensen's inequality, $L_p$ surface area, mean width, polytope, simplex, surface area}	
\maketitle

\begin{abstract}
Weighted cone-volume functionals are introduced for the convex polytopes in $\mathbb{R}^n$. For these functionals, geometric inequalities are proved and the equality conditions are characterized.  A variety of corollaries are derived, including extremal properties of the regular polytopes involving the $L_p$ surface area. Some applications to crystallography and quantum theory are  also  presented. 
\end{abstract}


\section{Introduction}

A fundamental question in convex and discrete geometry is to determine the arrangement of a finite set of points on a given sphere which optimizes some geometric functional. A notable example is the discrete isodiametric problem, which asks to determine, among all polytopes with a given number of vertices on the sphere, the one with maximal volume. Two problems which are closely related, and of equal significance, are determining the polytopes with a given number of vertices in the sphere that attain the greatest surface area, or the greatest mean width. These problems are not only interesting geometrically, but they also arise in a wide variety of applications, including crystallography \cite{DHL,Makovicky} and information theory \cite{KLZ2017, LitvakSMWC}.

Naturally, the optimal configurations in all of these problems tend to exhibit the highest degree of symmetry possible. For example, the regular tetrahedron, the regular octahedron and the regular icosahedron have the greatest volume, and greatest surface area, among all polytopes with 4, 6 or 12 vertices inscribed in the unit sphere  \cite{Toth-RegularFigures}. For background on the volume, surface area and mean width maximization problems, we refer the reader to, for example, \cite{BermanHanes1970,bezdek-langi, DHL,Toth-RegularFigures, Horvath, HL-2014,LitvakSMWC,Melnyk} and the references therein.

As generalizations of volume and surface area of polytopes in $\R^n$, we introduce the weighted cone-volume functionals, which are defined in terms of a concave or convex function applied to the facet heights or facet volumes. For a given number of points on the sphere, we determine the configurations of points that optimize these functionals in certain special cases, thereby generalizing several classical results on inscribed polytopes with maximum volume or maximum surface area.  The main results also yield a variety of corollaries on  polytopes inscribed in a sphere which maximize $L_p$ surface area. In particular, we show that the regular simplex maximizes the $L_p$ surface area (for $p\in[0,1]$) among all simplices inscribed in the sphere. For polytopes in $\R^3$, we also introduce generalizations of the mean width called the weighted edge-curvature functionals, where a concave or convex function is applied to the dihedral angles or edge lengths of the polytope. Sharp geometric inequalities are proved for these functionals and the equality conditions are characterized.

The $L_p$ surface area measure is an extension of the classical surface area measure on convex bodies in $\R^n$ that contain the origin in their interiors. It was originally defined  by Lutwak for $p>1$ in the groundbreaking work \cite{Lutwak93}, which was instrumental in the subsequent development of the $L_p$ Brunn-Minkowski-Firey theory. In the case that a polytope $Q$ is chosen for the body and $p\in\R$, the \emph{$L_p$ surface area} $S_p(Q)$ of $Q$ can be expressed as
\[
S_p(Q)=\sum_{F\in\mathcal{F}_{n-1}(Q)}\dist(o,F)^{1-p}\vol_{n-1}(F),
\]
where $\mathcal{F}_{n-1}(Q)$ is the set of facets ($(n-1)$-dimensional faces) of $Q$,  $\vol_{n-1}(F)$ is the $(n-1)$-dimensional volume of the facet $F$ and $\dist(o,F)$ is the distance from the origin $o$ to (the affine hull of) $F$. In particular, $S_1(Q)=\vol_{n-1}(\partial Q)$ is the classical surface area of $Q$ and $n^{-1}S_0(Q)=\vol_n(Q)$ is the $n$-dimensional volume of $Q$. For more background on the $L_p$ surface area measure and the $L_p$ Brunn-Minkowski-Firey theory, see Chapters 9.1 and 9.2 of Schneider's book \cite{SchneiderBook}, for example. 

In the past decade or so, the Orlicz-Brunn-Minkowski theory has  emerged as a generalization of the $L_p$ theory. It has been studied extensively and developed rapidly, and many of the results from the $L_p$ theory have been successfully translated into the Orlicz setting; see \cite{HLYZ2010,JianLu2019,  LudwigReitzner, LYZ2010-2, LYZ2010-1} and the references therein for some examples. In particular, an Orlicz extension of the $L_p$ surface area has been sought \cite{HP2014}, and Zou and Xiong \cite{ZX2014} gave the following definition. For a convex body $K$ in $\R^n$ that contains the origin in its interior and an increasing concave function $\phi:[0,\infty)\to[0,\infty)$ with $\phi(0)=0$, the \emph{Orlicz surface area} $S_{\phi,{\rm Orlicz}}(K)$ of $K$ was defined in \cite{ZX2014} by
\begin{equation}\label{OrliczSA}
S_{\phi,{\rm Orlicz}}(K)=\int_{\partial K}\phi\left(\frac{1}{\langle x,\nu_K(x)\rangle}\right)\langle x,\nu_K(x)\rangle\,{\rm d}S_K(x),
\end{equation}
where $\partial K$ is the boundary of $K$, $\nu_K$ is the Gauss map of $K$ and $S_K$ is the classical surface area measure of $K$. In particular, if $Q$ is a polytope in $\R^n$ that contains the origin $o$ in its interior, then \eqref{OrliczSA} becomes
\begin{equation}\label{Orlicz-SA}
S_{\phi,{\rm Orlicz}}(Q) = \sum_{F\in\mathcal{F}_{n-1}(Q)} \phi\left(\dist(o,F)^{-1}\right) \dist(o,F) \vol_{n-1}(F).
\end{equation}
Furthermore, when $\phi_p(t)=t^{p}$ is chosen in \eqref{Orlicz-SA} for $p\in\R$, one recovers the  $L_p$ surface area of $Q$. In particular, if $\phi_1(t)=t$ then $S_{\phi_1,{\rm Orlicz}}(Q)=S_1(Q)=\vol_{n-1}(\partial Q)$, and if $\phi_0(t)=1$ then    $n^{-1}S_{\phi_0,{\rm Orlicz}}(Q)=n^{-1}S_0(Q)=\vol_n(Q)$.  In this article, we introduce  similar (but distinct) Orlicz surface area type functionals on the polytopes in $\R^n$, and we prove  sharp geometric inequalities for them. To formulate our main results, we first state some definitions and notation.


\section{Definitions and notation}

Define the following classes of weight functions:
\begin{align*}
    {\rm Conc}_{(0,\infty)}&:=\left\{\varphi: (0,\infty)\to(0,\infty) \,\big| \, \varphi\text{ is concave}\right\}\\
    {\rm Conv}_{(0,\infty)}&:=\left\{\psi: (0,\infty)\to(0,\infty) \,\big| \, \psi\text{ is convex}\right\}.
\end{align*}
We will use the corresponding arrow to denote the subset of increasing (or decreasing) functions.  For example, $\varphi\in{\rm Conc}_{(0,\infty)}^{\uparrow}$ means that $\varphi\in{\rm Conc}_{(0,\infty)}$ and $\varphi$ is increasing, while $\psi\in{\rm Conv}_{(0,\infty)}^{\downarrow}$ means that $\psi\in{\rm Conv}_{(0,\infty)}$ and $\psi$ is decreasing. 

Let $\mathcal{P}_n$ denote the family of all convex polytopes in $\R^n$ which contain the origin $o$ in their interiors.  For $Q\in\mathcal{P}_n$, we denote the facets of $Q$ by $F_1,\ldots,F_{N_Q}$, where $N_Q:=|\mathcal{F}_{n-1}(Q)|$ is the number of facets of $Q$.  Also, let $r(Q)$ denote the radius of a largest sphere contained in $Q$ (note that such a sphere need not be unique), and let $R(Q)$ denote the radius of the smallest sphere containing $Q$. 
We will make use of the following auxiliary quantities associated to $Q\in\mathcal{P}_n$:
\begin{align*}
\overline{F_Q} &:= \dfrac{\vol_{n-1}(\partial Q)}{N_Q} \qquad \text{and} \qquad
\overline{h_Q} := \sum_{j=1}^{N_Q} \dist(o,F_j) \dfrac{\vol_{n-1}(F_j)}{\vol_{n-1}(\partial Q)}.
\end{align*}
We say that $Q$ is \emph{equiareal} if $\vol_{n-1}(F_j)=\overline{F_Q}$ is constant for all $j\in\{1,\ldots,N_Q\}$.  
Let $\mathcal{I}_n$ denote the set of polytopes in $\mathcal{P}_n$ which are inscribable in the unit sphere $\mathbb{S}^{n-1}=\{(x_1,\ldots,x_n)\in\R^n:\,\sum_{j=1}^n x_j^2=1\}$. We let $\mathcal{P}^{{\rm in}}_n$ be the subset of polytopes $Q$ in $\mathcal{P}_n$ that admit an insphere which is tangent to each facet of $Q$. For $Q\in\mathcal{P}^{{\rm in}}_n$, we let $r_Q$ and $i_Q$ denote the inradius and incenter of $Q$, respectively.

Our main results are concerned with the following Orlicz-type surface area functionals.

\begin{definition}\label{mainDef}
For a  (convex or concave) function $g:(0,\infty)\to(0,\infty)$, the $g$-\emph{weighted cone-volume functional} $S_g:\mathcal{P}_n\to(0,\infty)$ is defined by

\begin{equation*}
S_g(Q):=\sum_{j=1}^{N_Q} g(\dist(o,F_j))\vol_{n-1}(F_j).
\end{equation*}
\end{definition}

Some of the results in this paper will also be formulated for the weighted cone-volume functionals defined in the following remark.
\begin{remark}\label{starred}
One may similarly define the functional $S^{\rm in}_{g}:\mathcal{P}_n^{{\rm in}}\to(0,\infty)$ by 
\[S_{g}^{{\rm in}}(Q):=\sum_{j=1}^{N_Q} \dist(i_Q,F_j)\,g\left(\vol_{n-1}(F_j)\right).\]
\end{remark}

Each of the functionals in Definition \ref{mainDef} and Remark \ref{starred} may be viewed as weighted versions of the familiar \emph{cone-volume formula} applied to the polytope $Q\in\mathcal{P}_n$, namely,
\[
n\vol_{n}(Q) = \sum_{j=1}^{N_Q} \dist(x,F_j)\vol_{n-1}(F_j)
\]
where $x$ is any point in the interior of $Q$ and $\dist(x,F_j)$ is the distance from $x$ to the (affine hull of the) facet $F_j$. Moreover, choosing the function $\varphi_p(t):=t^{1-p}$ with $p\in[0,1]$ (respectively, $\psi_p(t)=t^{1-p}$ with $p\geq 1$ or $p<0$) in Definition \ref{mainDef}, we recover the $L_p$ surface area of $Q$:
\[
S_{\varphi_p}(Q)=S_p(Q)=\sum_{j=1}^{N_Q}\dist(o,F_j)^{1-p}\vol_{n-1}(F_j).
\]

Several other cone-volume type functionals on polytopes (and convex bodies) have been studied in the literature, including the $U$-functional,  which is closely connected to the famous LYZ conjectures. For more background, we refer the reader to, for example, \cite{BH-adv,HLL,Henk-Linke,LYZ-conevolume,SX2,Xiong-2010,Yanping-Binwu} and the references therein.

\section{Main results}

Our first result provides general bounds for the weighted cone-volume functionals of a polytope in terms of its surface area.

\begin{theorem}\label{mainThm}
Let $\varphi\in{\rm Conc}_{(0,\infty)}$.
\begin{itemize}
    \item[(i)] If $Q\in\mathcal{P}_n$, then $S_\varphi(Q) \leq \varphi\left(\overline{h_Q}\right) \vol_{n-1}\left(\partial Q\right)$. Equality holds if and only if $\varphi$ is affine or $Q\in\mathcal{P}^{{\rm in}}_n$ and the incenter of $Q$ is the origin. 
    
    \item[(ii)] If $Q\in\mathcal{P}^{{\rm in}}_n$, then  $S_{\varphi}^{{\rm in}}(Q)\leq \varphi\left(\overline{F_Q}\right)H_Q$ where  $H_Q:=\sum_{j=1}^{N_Q}\dist(i_Q,F_j)$.  Equality holds if and only if $\varphi$ is affine or $Q$ is equiareal.
\end{itemize}
If $\varphi$ is replaced by $\psi\in{\rm Conv}_{(0,\infty)}$ in (i) or (ii), then the reverse inequalities hold with the same equality conditions.
\end{theorem}

The regular simplex is a cornerstone of convex and discrete geometry, often arising as the solution to geometric extremal problems. For example,  among all  simplices inscribed in the unit sphere $\mathbb{S}^{n-1}$, the regular one has the greatest volume, and the greatest surface area. We prove a generalization in the following corollary.

\begin{corollary}\label{mainThm2}
Let $\varphi\in{\rm Conc}_{(0,\infty)}^{\uparrow}$ and let $T\in\mathcal{I}_n$ be a simplex. Then
    \begin{align*}\label{simplexineq}
     S_\varphi(T) &\leq  \dfrac{(n+1)^{\frac{n+1}{2}}}{n^{\frac{n}{2}-1}(n-1)!}\cdot\varphi\left(\frac{1}{n}\right) \qquad\text{ and } 
     \nonumber\\
     S^{{\rm in}}_\varphi(T) &\leq \dfrac{n+1}{n} \cdot\varphi\left(\dfrac{(n+1)^{\frac{n-1}{2}}}{n^{\frac{n}{2}-1}(n-1)!}\right).
    \end{align*}
 Equality holds in each case  if and only if $T$ is regular. 
\end{corollary}

Inequalities concerning polytopes combinatorially equivalent to the other regular polytopes in $\mathbb{R}^3$, analogous to those in Corollary \ref{mainThm2}, also hold. We omit their statements for brevity.

\begin{remark}\label{volume remark}
By letting $\varphi_p(t)=t^{1-p}$ with $p\in[0,1]$, the inequality for $S_{\varphi_p}(T)$ in Corollary \ref{mainThm2} represents an \emph{$L_p$ interpolation} of the following results.  Taking $p=0$, we recover the classical result which states that the regular simplex has greatest volume among all inscribed simplices. Choosing $p=1$, we recover the result of Tanner \cite{tanner} which states that the regular simplex has greatest surface area among all inscribed simplices. In general, if the volume and surface area maximizers coincide and admit an inradius, then this polytope is also the $L_p$ surface area maximizer for all $p\in[0,1]$. This is explained in more detail in Subsection \ref{littlewood-sec}.
\end{remark}

\begin{remark}\label{T-fn-rmk}
The weighted cone-volume functionals  are  related to the \emph{$T$-functional} of a polytope, introduced by Wieacker \cite{WieackerThesis} in stochastic geometry. For a polytope $Q\in\mathcal{P}_n$, parameters $a,b\geq 0$ and $j\in\{0,1,\ldots,n\}$, it is defined as \[T_{a,b}^{n,j}(Q)=\sum_{F\in\mathcal{F}_j(Q)} \dist(o,F)^a \vol_j(F)^b,\] where  $\mathcal{F}_j(Q)$ is the set of all $j$-dimensional faces of $Q$. The $T$-functional has been studied for various models of random polytopes; for some recent examples, see  \cite{HLRT-2022, KMTT-2019, KabluchkoEtAl2019}. Note that if $\varphi\in{\rm Conc}(0,\infty)$ is $a$-homogeneous, then
\[
S_\varphi(Q)=\varphi(1)\sum_{F\in\mathcal{F}_{n-1}(Q)}\dist(o,F)^a\vol_{n-1}(F)=\varphi(1)T_{a,1}^{n,n-1}(Q).
\]
In  particular, $T_{1-p,1}^{n,n-1}(Q)=S_p(Q)$. In the same way, if $\psi\in{\rm Conv}(0,\infty)$ is $a$-homogeneous, then $S_\psi(Q)=\psi(1)T_{a,1}^{n,n-1}(Q)$.
\end{remark}

\begin{remark}
    A result closely related to Corollary \ref{mainThm2} was shown very recently in \cite[Cor. 3.3]{Matzke-etal}. Let $T\in\mathcal{I}_n$ be a simplex and let $\triangle_n\in\mathcal{I}_n$ denote the regular inscribed simplex. Given $j\in\{1,\ldots,n\}$ and $s\in(0,2]$, we have $\sum_{F\in\mathcal{F}_j(T)}\vol_j(F)^s \leq \sum_{F'\in\mathcal{F}_j(\triangle_n)}\vol_j(F')^s$ with equality if and only if $T$ is regular.
\end{remark}

\section{Proofs of Theorem \ref{mainThm} and Corollary \ref{mainThm2}}

We will  use the following formulation of Jensen's inequality.
\begin{lemma}
 If $f$ is a concave function and $\lambda_1,\ldots,\lambda_N\geq 0$ satisfy $\sum_{i=1}^N \lambda_i=1$, then $\sum_{i=1}^N\lambda_i f(x_i) \leq f(\sum_{i=1}^N\lambda_i x_i)$. The inequality is reversed if $f$ is convex. Equality holds if and only if $x_1=\ldots=x_N$ or $f$ is affine. 
\end{lemma}

\subsection{Proof of Theorem \ref{mainThm}}

Let $Q\in\mathcal{P}_n$. Denote the facets of $Q$ by $F_1,\ldots, F_{N_Q}$ and set $h_j:=\dist(o,F_j)$.  Since $\vol_{n-1}(\partial Q) = \sum_{j=1}^{N_Q}\vol_{n-1}(F_j)$, we may express $\overline{h_Q}$ as the following convex combination of the $h_j$:
\[
\overline{h_Q} = \sum_{j=1}^{N_Q} \dfrac{\vol_{n-1}(F_j)}{\vol_{n-1}(\partial Q)} h_j.
\]
Thus Jensen's inequality yields
\[
S_\varphi(Q) \leq \varphi\left(\overline{h_Q}\right)\vol_{n-1}(\partial Q),
\]
with equality if and only if $\varphi$ is affine or all of the $h_j$ are the same. 
The latter condition implies that $h_j=r_Q$ for all $j$ and that the incenter of $Q$ is the origin (that is, $Q\in\mathcal{P}^{{\rm in}}_n$ and $i_Q=o$). For $Q\in\mathcal{P}^{{\rm in}}_n$, by Jensen's inequality we have
\begin{align*}
\dfrac{S^{{\rm in}}_{\varphi}(Q)}{H_Q} &= \sum_{j=1}^{N_Q}\dfrac{\dist(i_Q,F_j)}{H_Q}\varphi\left(\vol_{n-1}(F_j)\right)
\leq \varphi \left( \overline{F_Q} \right).
\end{align*}
Equality holds if and only if $\varphi$ is affine or $\vol_{n-1}(F_j) = \overline{F_Q}$ for all $j\in\{1,\ldots,N_Q\}$ (that is, $Q$ is equiareal). 

If $\varphi$ is replaced by $\psi\in{\rm Conv}_{(0,\infty)}$ in either of the above proofs, then in each case the direction of Jensen's inequality reverses and the same equality conditions hold. \qed

\vspace{3mm}

 To prove Corollary \ref{mainThm2}, we also need the following extremal property of the regular simplex in $\R^n$, which is called \emph{Euler's inequality}. A proof was given by L. Fejes T\'oth \cite[pp. 312--313]{Toth-RegularFigures}, where it was shown that the simplices of maximum and minimum volume inscribed and circumscribed, respectively, to a given sphere are regular.  Another short proof can be found in the article \cite{KT1979} by Klamkin and Tsintsifas, for example.

\begin{lemma}[Euler's inequality]\label{ratiolemma}
If $T$ is a simplex in $\R^n$ with inradius $r_T$ and circumradius $R_T$, then $R_T\geq nr_T$ with equality if and only if $T$ is regular.
\end{lemma}

\subsection{Proof of Corollary \ref{mainThm2}}
Let $T\in\mathcal{I}_n$ be any inscribed simplex. Since $T\in\mathcal{P}^{{\rm in}}_n$, we have $r_T=\overline{h_T}$. Since $R_T=1$, by Euler's inequality $r_T\leq r_{\triangle_n}= 1/n$ with equality  if and only if $T$ is regular. Thus, since $\varphi$ is increasing we have $\varphi(r_T)\leq\varphi(1/n)$ with equality if and only if $r_T=1/n$ (or equivalently, if and only if $T$ is regular). Let $\mathcal{T}_n$ denote the set of simplices in $\mathcal{I}_n$. Tanner \cite{tanner} showed that
\begin{equation}\label{max-SA-simplex}
\max_{T'\in\mathcal{T}_n}\vol_{n-1}(\partial T')= \vol_{n-1}(\partial\triangle_n)=\frac{(n+1)^{\frac{n+1}{2}}}{n^{\frac{n}{2}-1}(n-1)!}.
\end{equation}
Therefore, by Theorem \ref{mainThm}, Lemma \ref{ratiolemma} and  since $\varphi$ is increasing,
\begin{align*}
    S_\varphi(T) &\leq \varphi(r_T)\vol_{n-1}(\partial T)\\
    &\leq \varphi\Big(\max_{T'\in\mathcal{T}_n}r_{T'}\Big)\cdot\max_{T''\in\mathcal{T}_n}\vol_{n-1}(\partial T'')\\
    &=\varphi(r_{\triangle_n})\vol_{n-1}(\partial \triangle_n)\\
    &= \frac{(n+1)^{\frac{n+1}{2}}}{n^{\frac{n}{2}-1}(n-1)!}\cdot\varphi\left(\frac{1}{n}\right)
\end{align*}
with equality if and only if $\varphi$ is affine or $T$ is regular. Similarly, by Theorem \ref{mainThm}(ii), the fact that $\varphi$ is increasing,  Euler's inequality and \eqref{max-SA-simplex}, we obtain
\begin{align*}
S_{\varphi}^{{\rm in}}(T) &\leq \varphi\left(\frac{\vol_{n-1}(\partial T)}{n+1}\right)\sum_{j=1}^{n+1}\dist(i_T,F_j)\\
&=(n+1)\varphi\left(\frac{\vol_{n-1}(\partial T)}{n+1}\right)r_T\\
&\leq (n+1)\varphi\left(\frac{\max_{T'\in\mathcal{T}_n}\vol_{n-1}(\partial T')}{n+1}\right)\cdot\max_{T''\in\mathcal{T}_n}r_{T''}\\
&=(n+1)\varphi\left(\frac{\vol_{n-1}(\partial \triangle_n)}{n+1}\right)r_{\triangle_n}\\
&=\dfrac{n+1}{n} \cdot\varphi\left(\dfrac{(n+1)^{\frac{n-1}{2}}}{n^{\frac{n}{2}-1}(n-1)!}\right).
\end{align*}
Equality holds if and only if $\varphi$ is affine or $T$ is regular.  \qed

\section{Corollaries for polytopes in $\R^2$ and $\R^3$}\label{Applications}


\subsection{A planar result}

We begin in the plane with the following result.
\begin{corollary}
Let $Q\in\mathcal{P}^{{\rm in}}_2$ be a planar convex polygon with $v$ vertices. If $\varphi\in{\rm Conc}^{\uparrow}_{(0,\infty)}$, then 
\begin{equation}\label{planar-up}
   S_\varphi(Q)\leq 2vR(Q)\sin\frac{\pi }{v}\cdot\varphi\left(R(Q)\cos\frac{\pi}{v}\right).
\end{equation}
Equality holds if and only if $Q$ is regular with incenter at the origin.
\end{corollary}

\begin{proof}
By a result of L. Fejes T\'oth \cite{Toth-1948-article2, Toth-1948-article}, for any $P\in\mathcal{P}_2$ we have $r(P)/R(P)\leq \cos(\pi/v)$  with equality if and only if $P$ is regular.  Hence, by hypothesis, $r_Q\leq R(Q)\cos(\pi/v)$ with equality if and only if $Q$ is regular. Let $B_Q$ denote the circumball of $Q$. By Theorem \ref{mainThm} and the fact that $\varphi$ is increasing,
\begin{align*}
    S_\varphi(Q) &\leq \varphi(r_Q)\vol_1(\partial Q) \leq \varphi\left(R(Q)\cos\tfrac{\pi}{v}\right)\max_{P\in\mathcal{P}_2, P\subset B_Q}\vol_1(\partial P)
    =2vR(Q)\sin\tfrac{\pi }{v}\varphi\left(R(Q)\cos\tfrac{\pi}{v}\right).
\end{align*}
Equality is achieved in both inequalities simultaneously if and only if $Q$ is a  regular polygon with incenter at the origin. 
\end{proof}

\begin{remark}
For $\varphi_p(t)=t^{1-p}$ and $p\in[0,1]$, inequality \eqref{planar-up} may be regarded as an \emph{$L_p$ interpolation} of the classical inequalities which state that among all  polygons inscribed in a circle of given radius, the regular polygon maximizes the area ($p=0$) and maximizes the perimeter ($p=1$). For more details, please see Subsection \ref{littlewood-sec}.
\end{remark}


\subsection{Inequalities for the Platonic solids}

In this section, we will extend the following bounds for the surface area of a polytope to the weighted cone-volume functionals. In what follows, let $\omega_k:=\frac{\pi k}{6(k-2)}$.

\begin{lemma}\label{SA-bounds}
Let $Q\in\mathcal{P}_3$ have $v$ vertices, $e$ edges and $f$ facets.  Then the surface area of $Q$ may be bounded as follows:
\begin{itemize}
    \item[(i)]  $\vol_2(\partial Q) \geq e\sin\frac{\pi f}{e}(\tan^2\frac{\pi f}{2e}\tan^2\frac{\pi v}{2e}-1)r(Q)^2$ with equality if and only if $Q$ is regular;
    
    \item[(ii)] $\vol_2(\partial Q)\geq 6(f-2)\tan\omega_f(4\sin^2\omega_f-1)r(Q)^2$ with equality if and only if $Q$ is a regular tetrahedron, a regular hexahedron or a regular dodecahedron;
    
    \item[(iii)] $\vol_2(\partial Q)\leq \frac{3\sqrt{3}}{2}(v-2)(1-\frac{1}{3}\cot^2\omega_v) R(Q)^2$ with equality if and only if $Q$ is a regular tetrahedron, a regular octahedron or a regular icosahedron.

    \item[(iv)] If, in addition, $Q$ satisfies the foot condition where the foot of the perpendicular from the circumcenter of $Q$ to each facet-plane and each edge-line lies in the corresponding facet or edge, then $\vol_2(\partial Q) \leq e\sin\frac{\pi f}{e}(1-\cot^2\frac{\pi f}{2e}\cot^2\frac{\pi v}{2e})R(Q)^2$ with equality if and only if $Q$ is regular.
    \end{itemize}
\end{lemma}
\noindent Parts (i) and (iv) were shown in \cite{Toth1950} (see also \cite[pp. 154-155]{FejesToth} and \cite[p. 279]{Toth-RegularFigures}).  Part (ii) was shown by Fejes T\'oth in \cite{Toth-isep}. Part (iii) was proved by Linhart \cite{Linhart}, where it was also shown that the condition on the edges in the foot condition in (iv) is superfluous. 

 We will  need the following bounds on the \emph{spherical shell} $R(Q)/r(Q)$ of a convex polytope $Q$ in $\R^3$, which are due to Fejes T\'oth \cite[p. 264]{Toth-RegularFigures} (see also \cite{Toth-1943} and \cite[pp. 117, 131]{FejesToth}).

\begin{lemma}\label{shells-lemma}
Let $Q\in\mathcal{P}_3$ have $v$ vertices, $e$ edges and $f$ facets.  Then the spherical shell $R(Q)/r(Q)$ of $Q$ may be bounded as follows:
\begin{itemize}
    \item[(i)]  $\frac{R(Q)}{r(Q)}\geq \tan\frac{\pi f}{2e}\tan\frac{\pi v}{2e}$ with equality if and only if $Q$ is regular;
    
    \item[(ii)] $\frac{R(Q)}{r(Q)}\geq \sqrt{3}\tan\omega_v$ with equality if and only if $Q$ is a regular tetrahedron, a regular octahedron or a regular icosahedron.
\end{itemize}
\end{lemma}

Theorem \ref{mainThm} and the previous two lemmas lead to the following corollary.

\begin{corollary}\label{R3cor}
Let $Q\in\mathcal{P}^{{\rm in}}_3$ have $v$ vertices, $e$ edges and $f$ facets, and  let $\varphi\in {\rm Conc}^{\uparrow}_{(0,\infty)}$. 
\begin{itemize}
\item[(i)] We have  
\[
S_\varphi(Q) \leq \frac{3\sqrt{3}}{2}(v-2)\left(1-\frac{1}{3}\cot^2\omega_v\right)R(Q)^2\cdot\varphi\left(3^{-1/2}R(Q)\cot\omega_v\right)
\]
with equality 
if and only if $\varphi$ is affine or $Q$ is a regular tetrahedron, a regular octahedron or a regular icosahedron, and $Q$ has incenter at the origin. 

\item[(ii)] Suppose, in addition, that $Q$ satisfies the foot condition. Then
\[
S_\varphi(Q) \leq e\sin\frac{\pi f}{e}\left(1-\cot^2\frac{\pi f}{2e}\cot^2\frac{\pi v}{2e}\right)R(Q)^2\cdot\varphi\left(R(Q)\cot\frac{\pi f}{2e}\cot\frac{\pi v}{2e}\right)
\]
with equality if and only if $\varphi$ is affine or $Q$ is regular and  has incenter at the origin.
\end{itemize}
\end{corollary}

\begin{proof}
Part (i) follows from Theorem \ref{mainThm}(i), Lemma \ref{SA-bounds}(ii), Lemma \ref{shells-lemma}(ii) and the hypothesis that $\varphi$ is increasing. Equality holds if and only if equality holds in Theorem \ref{mainThm}(i), Lemma \ref{SA-bounds}(ii),  and Lemma \ref{shells-lemma}(ii), that is, if and only if $Q$ is a regular tetrahedron, a regular octahedron or a regular icosahedron with incenter at the origin. Part (ii) follows along the same lines, except that now we instead use Lemma \ref{SA-bounds}(iii) and Lemma \ref{shells-lemma}(i).
\end{proof}

\begin{remark}
The $v=4$ case in Corollary \ref{R3cor}(i) is the subject of Corollary \ref{mainThm2} (with $n=3$). The concrete bounds in (i) are also listed below for $v=6$ and $v=12$, respectively:
\begin{align*}
S_\varphi(Q)&\leq 4\sqrt{3}\cdot\varphi\left(\frac{1}{\sqrt{3}}\right)\qquad\text{and}\qquad
S_\varphi(Q)\leq (10-2\sqrt{5})\cdot\varphi\left(\frac{\sqrt{25+10\sqrt{5}}}{5\sqrt{3}}\right).
\end{align*}
\end{remark}

\begin{remark}
L. Fejes T\'oth proved in \cite[p. 264]{Toth-RegularFigures} that if $Q\in\mathcal{P}_3$ has $v$ vertices, then
\begin{equation}\label{Toth-vol-bd-vertices}
    \vol_3(Q)\leq \frac{1}{2}(v-2)\cot\omega_v\left(1-\cot^2\omega_v\right)R(Q)^3
\end{equation}
with equality if and only if $Q$ is a regular tetrahedron, a regular octahedron or a regular icosahedron. Thus, choosing $\varphi_p(t)=t^{1-p}$ with $p\in[0,1]$ in Corollary \ref{R3cor}(i), we obtain an \emph{$L_p$ interpolation} of Lemma \ref{SA-bounds}(iii) and \eqref{Toth-vol-bd-vertices}, but with stricter equality conditions. Choosing $p=0$, we recover \eqref{Toth-vol-bd-vertices}; choosing $p=1$, we recover Lemma \ref{SA-bounds}(iii).  Suppose additionally that $Q$ has $e$ edges and $f$ facets. Fejes T\'oth \cite[p. 263]{Toth-RegularFigures} proved that
\begin{equation}\label{Toth-vol-bd-vef}
    \vol_3(Q)\leq \frac{2e}{3}\cos^2\frac{\pi f}{2e}\cot\frac{\pi v}{2e}\left(1-\cot^2\frac{\pi f}{2e}\cot^2\frac{\pi v}{2e}\right)R(Q)^3,
\end{equation}
with equality if and only if $Q$ is regular. If we further assume that $Q$ satisfies the foot condition, then choosing $\varphi_p(t)=t^{1-p}$ with $p\in[0,1]$ in Corollary \ref{R3cor}(ii), we obtain an $L_p$ interpolation of Lemma \ref{SA-bounds}(iv) and \eqref{Toth-vol-bd-vef}. There is a difference in the equality conditions of these results since those of Corollary \ref{R3cor} also require that $Q$ has incenter at the origin.
\end{remark}

\begin{corollary}\label{veflower}
Suppose that  $Q\in\mathcal{P}^{{\rm in}}_3$ has  $v$ vertices, $e$ edges, $f$ facets. Let $\psi\in{\rm Conv}_{(0,\infty)}$. 
\begin{itemize}
\item[(i)] We have
    \[
    S_\psi(Q) \geq e\sin\frac{\pi f}{e}\left(\tan^2\frac{\pi f}{2e}\tan^2\frac{\pi v}{2e}-1\right)r_Q^2\,\psi(r_Q)
    \]
with equality if and only if $Q$ is regular and has incenter at the origin.

\item[(ii)] We have
\[
S_\psi(Q) \geq 6(f-2)\tan\omega_f(4\sin^2\omega_f-1)r_Q^2\,\psi(r_Q)
\]
with equality if and only if $Q$ is a regular tetrahedron, a regular hexahedron or a regular dodecahedron, and $Q$ has incenter at the origin.
\end{itemize}
\end{corollary}

\begin{remark}
Regarding volume bounds, L. Fejes T\'oth \cite[p. 263]{Toth-RegularFigures} also proved that
\begin{align}
\vol_3(Q) &\geq \frac{e}{3}\sin\frac{\pi e}{f}\left(\tan^2\frac{\pi f}{2e}\tan^2\frac{\pi v}{2e}-1\right)r(Q)^3 \label{Toth-vol-lower-vef}\\
    \vol_3(Q) &\geq (f-2)\sin2\omega_f\left(3\tan^2\omega_f-1\right)r(Q)^3. \label{Toth-vol-lower-f}
\end{align}
Equality holds in \eqref{Toth-vol-lower-vef} if and only if $Q$ is regular; equality holds in \eqref{Toth-vol-lower-f} if and only if $Q$ is a regular tetrahedron, a regular hexahedron or a regular dodecahedron. With a special choice of functions, we can recover the corresponding volume and surface area estimates from Corollary \ref{veflower}, but with stricter equality conditions since Corollary \ref{veflower} also requires that $Q$ has incenter at the origin. Thus (up to this additional condition), choosing $\psi(t)=1$ in Corollary \ref{veflower}(i), we recover Lemma \ref{SA-bounds}(i), and choosing $\psi(t)=t$ in Corollary \ref{veflower}(i), we recover \eqref{Toth-vol-lower-vef}. Similarly, choosing $\psi(t)=1$ in Corollary \ref{veflower}(ii), we recover Lemma \ref{SA-bounds}(ii), and choosing $\psi(t)=t$ in Corollary \ref{veflower}(ii), we recover \eqref{Toth-vol-lower-f}. Again, there is a difference in the equality conditions of these results. 
\end{remark}

\section{A weighted edge curvature type functional for polytopes in $\R^3$}

The \emph{mean width} $W(C)$ of a convex body $C$ in $\R^n$ is defined by
\[
W(C) = 2\int_{\Sp}h_C(u)\,d\sigma(u),
\]
where $h_C(u)=\max_{x\in C}\langle x,u\rangle$ is the support function of $C$ in the direction $u\in\Sp$ and $\sigma$ is the uniform probability measure on $\Sp$. In the special case $C=Q$ is a polytope in $\mathbb{R}^3$ with edge set $\mathcal{F}_1(Q)$, the \emph{edge curvature} $M(Q)$ of $Q$ is defined as (see, for example, \cite[p. 278]{Toth-RegularFigures})
\begin{equation}
    M(Q) = \frac{1}{2}\sum_{E\in\mathcal{F}_1(Q)}\vol_1(E)\theta_E.
\end{equation}
Here $\theta_E$ is the external angle of $Q$ at $E$, that is, the angle determined by the outer unit normals of the two facets of $Q$ that meet at $E$  (in other words, $\theta_E$ is $2\pi$ minus the dihedral angle of $E$). Note that $M(Q)=\pi W(Q)$. Let $\Lambda_Q:=\sum_{E\in\mathcal{F}_1(Q)}\vol_1(E)$ denote the total edge length of $Q$ and set  $\overline{\theta_Q}:=\sum_{E\in\mathcal{F}_1(Q)}\frac{\vol_1(E)}{\Lambda_Q}\cdot \theta_E$. For a  (convex or concave) function $g:(0,\infty)\to(0,\infty)$, we define the $g$-\emph{weighted edge curvature functional} $M_g:\mathcal{P}_3\to(0,\infty)$ by 
\begin{align*}
    M_g(Q) :=\frac{1}{2}\sum_{E\in\mathcal{F}_1(Q)}\vol_1(E)g(\theta_E).
\end{align*}
One may also define $M_g^{{\rm in}}(Q)$ in a similar fashion by instead applying $g$ to $\vol_1(E)$ in the sum. A simple inequality for the weighted edge curvature functionals follows.

\begin{theorem}\label{MWthm}
Let $\varphi\in{\rm Conc}_{(0,\infty)}$. If $Q$ is a convex polytope in $\R^3$, then $M_\varphi(Q) \leq \varphi\left(\overline{\theta_Q}\right)\Lambda_Q$ with equality if and only if $\varphi$ is affine or all of the dihedral angles of $Q$ are equal. If $\varphi\in{\rm Conc}_{(0,\infty)}$ is replaced by $\psi\in{\rm Conv}_{(0,\infty)}$, then the reverse inequality holds  with the same equality conditions.
\end{theorem}
\noindent The proof follows along the same lines as that of Theorem \ref{mainThm}. We leave the details to the interested reader. 


\section{Comparing volume and surface area maximizers}

\subsection{An interpolation result for the $L_p$ surface area}\label{littlewood-sec}

For a fixed integer $v\geq n+1$, set
\[
\mathcal{P}_{n,v}:=\{Q\in\mathcal{P}_n: Q\text{ has at most }v\text{ vertices}\}
\]
and $\mathcal{P}^{{\rm in}}_{n,v}:=\mathcal{P}_{n,v}\cap\mathcal{P}^{{\rm in}}_n$. By choosing the function  $\varphi_p(t)=t^{1-p}$ with $p\in[0,1]$, we can view Theorem \ref{mainThm} as an $L_p$ interpolation between the classical problems of maximizing the surface area and maximizing the volume.  We make this explicit in the following lemma.
\begin{lemma}\label{littlewood}
Suppose that $\widehat{Q}\in\mathcal{P}_{n,v}^{{\rm in}}\cap\mathcal{I}_n$ maximizes both the volume and the surface area functional over $\mathcal{P}_{n,v}\cap\mathcal{I}_n$.  Then for any $p\in[0,1]$, $\widehat{Q}$ also maximizes $S_p$ over $\mathcal{P}_{n,v}\cap\mathcal{I}_n$. 
\end{lemma}

\begin{proof}
Let $Q\in\mathcal{P}_{n,v}\cap\mathcal{I}_n$.  By H\"older's inequality,
\begin{align*}
    S_{p}(Q) &= \sum_{j=1}^{N_Q}\dist(o,F_j) ^{1-p}\vol_{n-1}(F_j)^{(1-p)+p}\\
    &\leq \left( \sum_{j=1}^{N_Q} \dist(o,F_j)\vol_{n-1}(F_j)\right)^{1-p}\left(\sum_{j=1}^{N_Q}\vol_{n-1}(F_j)\right)^{p}\\
    &=S_{0}(Q)^{1-p}S_{1}(Q)^{p}\\
    &\leq S_{0}(\widehat{Q})^{1-p}S_{1}(\widehat{Q})^{p}.
\end{align*}
Since $\widehat{Q}\in\mathcal{P}_{n,v}^{{\rm in}}$, the last line may be simplified to
\[
S_{0}(\widehat{Q})^{1-p}S_{1}(\widehat{Q})^{p} = r_{\widehat{Q}}^{1-p}\vol_{n-1}(\partial \widehat{Q}) = S_p(\widehat{Q}).
\]
The result follows.
\end{proof}

When $n=3$ and $v\geq 4$ is fixed, the volume and surface area maximizers in $\mathcal{P}_{3,v}\cap \mathcal{I}_3$ coincide for $v\in\{4,5,6,12\}$, as can be seen by comparing the results in \cite{BermanHanes1970, DHL, Toth-RegularFigures}. For $v=4$, $v=6$ and $v=12$, the maximizers are the regular tetrahedron, the regular octahedron and the regular icosahedron, respectively. For $v=5$, the maximizer is a triangular bipyramid with two vertices at the north and south poles $\pm e_3$ and the other three forming an equilateral triangle in the equator $\mathbb{S}^2\cap(\spann(e_3))^\perp$ \cite{DHL}. Moreover, the maximum volume and maximum surface area bipyramids in $\mathcal{P}_{3,v}\cap\mathcal{I}_3$ coincide \cite{DHL}.  Thus by Lemma \ref{littlewood}, we obtain the following result.
\begin{corollary}
Let $Q\in\mathcal{P}_{3,v}\cap\mathcal{I}_3$ be a bipyramid and let $p\in[0,1]$. Then
\[
S_p(Q) \leq 2(v-2)\sin\frac{\pi}{v-2}\left(\cos\frac{\pi}{v-2}\right)^{1-p}\left(1+\cos^2\frac{\pi}{v-2}\right)^{p/2}
\]
with equality if and only if $Q$ has two vertices at the north and south poles $\pm e_3$ and the remaining $v-2$ vertices form a regular $(v-2)$-gon in the equator $\mathbb{S}^2\cap(\spann(e_3))^\perp$.
\end{corollary}

\begin{proof}
It was shown in \cite[Cor. 2]{DHL} that $\vol_2(\partial Q)\leq 2(v-2)\sqrt{1+\cos^2\frac{\pi}{v-2}}\cdot\sin\frac{\pi}{v-2}$, and in \cite[p. 73]{DHL} that $\vol_3(Q)\leq \frac{1}{3}(v-2)\sin\frac{2\pi}{v-2}$. Equality holds in both cases if and only if $Q$ has two vertices at the north and south poles $\pm e_3$ and the remaining $v-2$ vertices form a regular $(v-2)$-gon in the equator $\mathbb{S}^2\cap(\spann(e_3))^\perp$. Let $\widehat{Q}_v$ denote this bipyramid. Then the inradius $r_{\widehat{Q}_v}$ is given by
\[
r_{\widehat{Q}_v}=\frac{3\vol_3(\widehat{Q}_v)}{\vol_2(\partial\widehat{Q}_v)}=\frac{\cos\frac{\pi}{v-2}}{\sqrt{1+\cos^2\frac{\pi}{v-2}}}.
\]
By Lemma \ref{littlewood}, for any $p\in[0,1]$ we have 
\[
S_p(Q)\leq S_p(\widehat{Q}_v)=r_{\widehat{Q}_v}^{1-p}\vol_2(\partial \widehat{Q}_v).
\]
The result follows from simple computations.
\end{proof}


\subsection{Volume and surface area maximizers need not coincide}

In general, the volume and surface area maximizers in $\mathcal{P}_{n,v}\cap\mathcal{I}_n$ need not coincide. For $n=3$ and $v=8$, the volume maximizer was determined by Berman and Hanes \cite{BermanHanes1970}. It is unique (up to rotations) and  its vertices $p_1,\ldots,p_8$ are given by
\begin{alignat*}{2}
    p_1 &=(\sin 3\vartheta, 0,\cos 3\vartheta), \qquad &&p_5=(0,-\sin 3\vartheta,-\cos 3\vartheta),\\
    p_2 &=(\sin \vartheta, 0,\cos \vartheta), \qquad &&p_6=(0,-\sin \vartheta,-\cos \vartheta),\\
    p_3 &=(-\sin \vartheta, 0,\cos \vartheta), \qquad &&p_7=(0,\sin \vartheta,-\cos \vartheta),\\
    p_4 &=(-\sin 3\vartheta, 0,\cos 3\vartheta), \qquad &&p_8=(0,\sin 3\vartheta,-\cos 3\vartheta),
\end{alignat*}
where $\vartheta:=\arccos\left(\sqrt{(15+\sqrt{145})/40}\right)$. This polytope is shown in the figure below, which is a recreation of \cite[Fig. 1]{BermanHanes1970}. As one can check, replacing the angle parameter $\vartheta$ in each of these eight points by $\theta:=0.62$  yields a combinatorially equivalent polytope with greater surface area. This increases the surface area from approximately $8.11757$ to approximately $8.11978$.
 
 \begin{center}
\tdplotsetmaincoords{90}{60}
\def\r{1}
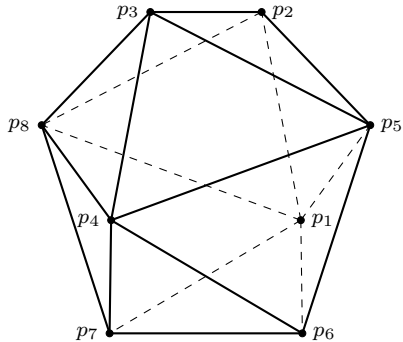
\begin{tikzpicture}[scale=2.6,line join=bevel, tdplot_main_coords]
    \coordinate (O) at (0,0,0);

\coordinate (P1) at ({sin(3*34.6927133)},0,{cos(3*34.6927133)});
\coordinate (P2) at ({sin(34.6927133)},0,{cos(34.6927133)});
\coordinate (P3) at ({-sin(34.6927133)},0,{cos(34.6927133)});
\coordinate (P4) at ({-sin(3*34.6927133)},0,{cos(3*34.6927133)});
\coordinate (P5) at (0,{-sin(3*34.6927133 )},{-cos(3*34.6927133)});
\coordinate (P6) at (0,{-sin(34.6927133 )},{-cos(34.6927133)});
\coordinate (P7) at (0,{sin(34.6927133 )},{-cos(34.6927133)});
\coordinate (P8) at (0,{sin(3*34.6927133 )},{-cos(3*34.6927133)});

\begin{scope}[thick]
    \draw[] (P2) -- (P3) node[below,pos=0.8,xshift=4]{};
      \draw[] (P3)--(P5)node[below,pos=0.3,xshift=-4]{};
      \draw[] (P3)--(P4) node[below,pos=0.4,xshift=-3]{};
      \draw[] (P3)--(P8) node[below,pos=0.6,xshift=3]{};
      \draw[] (P4)--(P5) node[below,pos=0.6,xshift=3]{};
       \draw[] (P4)--(P6) node[below,pos=0.6,xshift=3]{};
        \draw[] (P4)--(P7) node[below,pos=0.6,xshift=3]{};
         \draw[] (P4)--(P8) node[below,pos=0.6,xshift=3]{};
          \draw[] (P5)--(P6) node[below,pos=0.6,xshift=3]{};
           \draw[] (P6)--(P7) node[below,pos=0.6,xshift=3]{};
            \draw[] (P7)--(P8) node[below,pos=0.6,xshift=3]{};
            \draw[] (P2) -- (P8) node[below,pos=0.6]{};
\end{scope} 

\begin{scope}[dashed] 
    \draw[] (P1) -- (P2) node[below,pos=0.6]{};
    \draw[] (P1) -- (P5) node[below,pos=0.6]{};
    \draw[] (P1) -- (P6) node[below,pos=0.6]{};
    \draw[] (P1) -- (P7) node[below,pos=0.6]{};
    \draw[] (P1) -- (P8) node[below,pos=0.6]{};
     \draw[] (P2)--(P5) node[below,pos=0.3,xshift=4]{};
\end{scope}

\filldraw[black] (P1) circle (0.5pt) node[anchor=west]{\tiny $p_1$};
\filldraw[black] (P2) circle (0.5pt) node[anchor=west]{\tiny $p_2$};
\filldraw[black] (P3) circle (0.5pt) node[anchor=east]{\tiny $p_3$};
\filldraw[black] (P4) circle (0.5pt) node[anchor=east]{\tiny $p_4$};
\filldraw[black] (P5) circle (0.5pt) node[anchor=east]{\tiny $p_8$};
\filldraw[black] (P6) circle (0.5pt) node[anchor=east]{\tiny $p_7$};
\filldraw[black] (P7) circle (0.5pt) node[anchor=west]{\tiny $p_6$};
\filldraw[black] (P8) circle (0.5pt) node[anchor=west]{\tiny $p_5$};
 \coordinate (L1) at (0,0,-1);
    \node[] at (L1) {};
\end{tikzpicture} 
\end{center}

\subsection{Applications}
In the context of quantum theory, it was conjectured by Kazakov \cite[Conj. 4.2]{kazakov-thesis} that the volume and surface area maximizers coincide (see also the discussion in  \cite{DHL}). The counterexample above shows that this is not always the case. It also shows that the volume and surface area discrepancies from crystallography are, in fact, distinct measures of distortion  for the coordination polyhedra of ligand atoms on the unit sphere \cite{DHL, Makovicky} when considering \emph{global} maximizers. More specifically, it was shown in \cite{DHL} that the volume and surface area discrepancies are distinct measures of distortion on the class of pyramids in $\mathcal{P}_{3,v}\cap\mathcal{I}_3$. Furthermore, the argument above shows  that the  volume and surface area discrepancies are, in general, distinct notions of measure \emph{globally} for all polytopes in $\mathcal{P}_{3,v}\cap\mathcal{I}_3$.
 
 \subsection{Open problems}
 
 We conclude by highlighting a couple of related  problems, which, to the best of our knowledge, are open.
 
 \begin{problem}\label{SA-problem}
 Determine the maximum surface area polytope in $\mathcal{P}_{3,v}\cap\mathcal{I}_3$ for $v\in\{7,8,9,10,11\}$ and $v\geq 13$.
 \end{problem}
 The global surface area maximizers among all polytopes with $v$ vertices inscribed in $\mathbb{S}^2$ have been determined analytically for the cases $v=4,5,6$ and $12$ \cite{DHL, Toth-RegularFigures,Heppes, Krammer}. 
 For $v=7$, we conjecture that the maximizer is a pentagonal bipyramid with apexes at the north and south poles $\pm e_3$ and five more forming a regular pentagon in the equator $\mathbb{S}^2\cap e_3^\perp$ (see also \cite[Conj. 1.2]{FHLPW-2022}). The maximizer among all polytopes with 7 vertices that have congruent isosceles or congruent equilateral triangular facets was recently determined in \cite{FHLPW-2022}.
 
  \begin{problem}
 Determine the maximum volume polytope in $\mathcal{P}_{3,v}\cap\mathcal{I}_3$ for $v\in\{9,10,11\}$ and $v\geq 13$.
 \end{problem}
 The global volume maximizers among all polytopes with at most 8 vertices inscribed in $\mathbb{S}^2$ can be found in the article \cite{BermanHanes1970} by Berman and Hanes, and the case $v=12$ can be found in the book \cite{Toth-RegularFigures} by L. Fejes T\'oth.
 
 \section*{Acknowledgments}

SH would like to thank SIAM for registration and travel support provided by the National Science Foundation (NSF grant DMS – 1757085) to present this work at its Discrete Mathematics conference (DM22) held at Carnegie Mellon University in June of 2022. SH would  like to thank  the Georgia Institute of Technology and the organizers of the {\it Workshop in Convexity and High-Dimensional Probability}, held in May of 2022, for their hospitality (supported by NSF CAREER DMS-1753260). During the  workshop,   significant progress on  this manuscript was achieved. 


\bibliographystyle{plain}
\bibliography{main}

\vspace{3mm}

\noindent {\sc Department of Mathematics \& Computer Science, Longwood University, U.S.A.}

\noindent {\it E-mail address:} {\tt hoehnersd@longwood.edu}

\noindent {\it E-mail address:} {\tt ledfordjp@longwood.edu}

\end{document}